\documentclass[11pt,longtable,letterpaper,oneside,reqno]{amsart}
\usepackage{amsmath,amsfonts,amsthm,amssymb,amscd,mathtools,stmaryrd,url}
\usepackage{bbm}
\usepackage{epsfig}
\usepackage{graphicx}
\usepackage{latexsym}
\usepackage{longtable}
\usepackage{float}
\usepackage{hyperref}
\usepackage{color}

\DeclareFontFamily{OT1}{rsfs}{}
\DeclareFontShape{OT1}{rsfs}{n}{it}{<-> rsfs10}{}
\DeclareMathAlphabet{\mathscr}{OT1}{rsfs}{n}{it}

\newtheorem{prop}{Proposition}[section]
\newtheorem{thm}[prop]{Theorem}

\newtheorem{lem}[prop]{Lemma}
\newtheorem{defn}{Definition}

\newtheorem {defn }{Definition}

\numberwithin{equation}{section}
\usepackage[margin=1in]{geometry}
\usepackage{lipsum}

\begin{document}
\title{On a conjecture of Deaconescu}
\author[E. Hasanalizade]{Elchin Hasanalizade}
\address{Department of Mathematics and Computer Science, University of Lethbridge, 4401 University Drive, Lethbridge, Alberta, T1K 3M4 Canada}
\email{e.hasanalizade@uleth.ca}
\begin{abstract}
In 2000 Deaconescu raised a question whether there exists a composite $n$ for which $S_2(n)|\phi(n)-1$, where $\phi(n)$ is Euler's function and $S_2(n)$ is Schemmel's totient function. In this paper we prove that any such $n$ is odd, squarefree and has at least seven distinct prime factors. We also prove that any such $n$ with exactly $K$ distinct prime divisors is necessarily less than $2^{2^{K+1}}$.
\end{abstract}

\date{\today}
\maketitle

\section{Introduction}

Let $\phi$ denote Euler's totient function. In 1932 Lehmer \cite{L} conjectured that if $\phi(n)|n-1$, then $n$ has to be a prime number. A composite positive integer satisfying that divisibility is called Lehmer number or number with Lehmer property. Although this problem has not been settled so far, several partial results are known. Lehmer himself proved that if $n$ has Lehmer property then $n$ is odd, squarefree and has at least seven distinct prime factors. Cohen and Hagis \cite{CH} using computational methods established that $\omega(n)\ge14$, where $\omega(n)$ denotes the number of distinct prime divisors of $n$. Burcsi et al. \cite{BCF} showed if additionally $3|n$, then $\omega(n)\ge40\cdot10^6$ and $n>10^{36\cdot10^7}$. On the other hand, Pomerance \cite{P} proved that every Lehmer number $n$ is $<K^{2^K}$, where $K=\omega(n)$. Recently, Burek and {\.Z}mija \cite{BZ} have improved this upper bound to $2^{2^K}-2^{2^{K-1}}$.

In 2000 Deaconescu \cite{D} conjectured that for $n\ge2$
\begin{align*}
S_2(n)|\phi(n)-1
\end{align*}
if and only if $n$ is prime, where $S_2(n)$ is Schemmel's totient function defined by 
\begin{align*}
S_2(n)=n\prod_{p|n}\big(1-\frac{2}{p}\big).
\end{align*}

This problem seems to be as challenging as Lehmer's problem. Clearly, the conjecture states that for every $M\ge1$, the set $D_M$ of integers satisfying 
\begin{align}
\label{Eq1}
MS_2(n)=\phi(n)-1
\end{align}
contains only prime numbers. We say that a composite integer $n$ is a Deaconescu number (or has the Deaconescu property) if it satisfies \eqref{Eq1}.

In this short note we prove the following results.

\begin{thm}
\label{Thm1}
If $n$ is a Deaconescu number, then $n$ is odd, squarefree and $\omega(n)\ge7$.
\end{thm}

Inspired by the work of Burek and {\.Z}mija we will also get an upper bound for Deaconescu numbers.
\begin{thm}
\label{Thm2}
If $n$ has the Deaconescu property, then 
\begin{align*}
n<2^{2^K+K}-2^{2^{K-1}+K},
\end{align*}
where $K=\omega(n)$.
\end{thm}

Hernandez and Luca \cite{HL} proved that there are at most finitely many Lehmer numbers $n$ such that $P(\phi(n))\equiv0\ (\text{mod} \ n)$, where $P(X)\in\mathbb{Z}[X]$ is any monic non-constant polynomial. We will prove the analogous result for Deaconescu numbers.
\begin{thm}
\label{Thm3}
Let $P(X)\in\mathbb{Z}[X]$ be a monic non-constant polynomial. Then there are at most finitely many composite integers $n$ such that $S_2(n)|\phi(n)-1$ and $P(S_2(n))\equiv0\ (\text{mod} \ \phi(n))$.
\end{thm}

\section{Preliminaries}

In this section we shall collect some preliminary results. First, we give a group-theoretic interpretation of Schemmel's totient function.

\begin{defn}
Ler $R$ be a commutative ring with identity and $R^*$ be the multiplicative group of its units. A unit $u\in R$ is called exceptional if $1-u\in R^*$.
\end{defn}

Let $R^{**}$ denote the set of all exceptional units in $R$. In particular if $R=\mathbb{Z}_n$ the ring of residue classes mod $n$, then by definition we have 
\begin{align*}
\mathbb{Z}^{**}_n=\{a\in\mathbb{Z}_n: \ \text{gcd}(a,n)=1 \ \text{and} \ \text{gcd}(a-1,n)=1 \}.
\end{align*}

In 2010, Harrington and Jones \cite{HJ} proved that 
\begin{align*}
|\mathbb{Z}^{**}_n|=S_2(n).
\end{align*}

Note that $\phi(n)$ is always even if $n>2$. Thus $M$ in \eqref{Eq1} must be odd.

\begin{lem}
$n\in D_1$ if and only if $n=p$ for some prime $p$.
\end{lem}

\begin{proof}
If $n=p$, it is in $D_1$. If $n=p^{\alpha_1}_1\ldots p^{\alpha_r}_r$ ($r>1$) with $p_1<p_2<\ldots p_r$ then $S_2(n)<\phi(n)-1$ since $1\in\mathbb{Z}^{*}_n$ and $p_1+1\in\mathbb{Z}^{*}_n$ but $1\notin\mathbb{Z}^{**}_n$ and $p_1+1\notin\mathbb{Z}^{**}_n$.
\end{proof}

From now on we assume that $M>1$, $n$ always denotes an integer greater than 1 in $D_M$ for some $M>1$. Then we have
\begin{align}
\label{Eq2}
\frac{\phi(n)}{S_2(n)}>M\ge3.
\end{align}

The next lemma due to Nielsen \cite{N} plays important role in the proof of the upper bound for numbers with Deaconescu property.
\begin{lem}
Let $r,a,b\in\mathbb{N}$ and $x_1,\ldots,x_r$ be integers such that $1<x_1<x_2<\ldots<x_r$ and 
\begin{align}
\label{Eq3}
\prod_{j=1}^r \bigg(1-\frac{1}{x_j}\bigg)\le \frac{a}{b}<\prod_{j=1}^{r-1} \bigg(1-\frac{1}{x_j}\bigg).
\end{align}
Then 
\begin{align}
\label{Eq4}
a\prod_{j=1}^r x_j\le (a+1)^{2^r}-(a+1)^{2^{r-1}}.
\end{align}
\end{lem}

\section{Proofs}

\begin{proof}[Proof of Theorem \ref{Thm1}]
By definition of $S_2(n)$ it is clear that $n$ must be odd. If $n$ is not squarefree then $n$ has a prime factor $p_i$ for which $p_i|\phi(n)$ and $p_i|S_2(n)$. In this case if $n$ is Deaconescu number then $p_i|1$ which is impossible. Next, we show that $\omega(n)\ne2$. When $n=p_1p_2$ equation \eqref{Eq1} becomes
\begin{align*}
M(p_1-2)(p_2-2)=(p_1-1)(p_2-1)-1 \ \text{or} \ M-1=\frac{1}{p_1-2}+\frac{1}{p_2-2}.
\end{align*}
Hence $0<M-1\le1+\frac{1}{3}=\frac{4}{3}$ and $M=2$. which is impossible since $M$ is odd. If $2<\omega(n)\le6$, then $M=3$ and $3|n$. Indeed, if $n=p_1p_2\ldots p_r$ with $p_1<p_2<\ldots<p_r$ then \eqref{Eq2} gives $M<\frac{\phi(n)}{S_2(n)}\le\prod_{i=1}^r\frac{q_i-1}{q_i-2}=Q_r$, where $\{q_i\}=\{3,5,7,\ldots\}$ denotes the sequence of all odd primes. Since $Q_r<5$ for $2<r\le6$, we get $M=3$. If $2<r\le6$ and $3\nmid n$, then 
\begin{align*}
\frac{\phi(n)}{S_2(n)}=\prod_{i=1}^r\frac{p_i-1}{p_i-2}\le\prod_{i=1}^6\frac{q_{i+1}-1}{q_{i+1}-2}<3
\end{align*}
contradicting \eqref{Eq2}. Hence $3|n$. But if $n=3p_2\ldots p_r$ and $M=3$ equation \eqref{Eq1} becomes
\begin{align*}
3(p_2-2)\cdots(p_r-2)=2(p_2-1)\cdots(p_r-1)-1.
\end{align*}
Taking this equation modulo 3 we see that it has no solutions in primes. Hence $\omega(n)\ge7$.
\end{proof}

\begin{proof}[Proof of Theorem 1.2]
Let us write $n=p_1\cdots p_K$ where $p_1<p_2<\ldots<p_K$. Then 
\begin{align*}
\prod_{j=1}^K \bigg(1-\frac{1}{p_j-1}\bigg)=\frac{S_2(n)}{\phi(n)}<\frac{S_2(n)}{\phi(n)-1}.
\end{align*}
Moreover, 
\begin{align*}
\frac{\frac{S_2(n)}{\phi(n)-1}}{\prod\limits_{j=1}^{K-1} \bigg(1-\frac{1}{p_j-1}\bigg)}&=\frac{n\prod\limits_{j=1}^K \bigg(1-\frac{2}{p_j}\bigg)}{(\phi(n)-1)\prod\limits_{j=1}^{K-1} \bigg(1-\frac{1}{p_j-1}\bigg)}\\ 
&=\frac{n\prod\limits_{j=1}^{K-1} \bigg(1-\frac{2}{p_j}\bigg)\bigg(1-\frac{2}{p_K}\bigg)}{(\phi(n)-1)\prod\limits_{j=1}^{K-1} \bigg(1-\frac{1}{p_j-1}\bigg)}=\frac{n\prod\limits_{j=1}^{K-1} \bigg(1-\frac{1}{p_j}\bigg)\bigg(1-\frac{2}{p_K}\bigg)}{\phi(n)-1}\\ 
&=\frac{\phi(n)}{\phi(n)-1}\cdot\frac{\bigg(1-\frac{2}{p_K}\bigg)}{\bigg(1-\frac{1}{p_K}\bigg)}<1
\end{align*}
since $\phi(n)>p_K-1$. Thus 
\begin{align*}
\prod\limits_{j=1}^{K} \bigg(1-\frac{1}{p_j-1}\bigg)<\frac{S_2(n)}{\phi(n)-1}<\prod\limits_{j=1}^{K-1} \bigg(1-\frac{1}{p_j-1}\bigg).
\end{align*}
Hence, the inequality \eqref{Eq3} is satisfied for $x_j=p_j-1$, $r=K$, $a=1$, $b=\frac{\phi(n)-1}{S_2(n)}$. From \eqref{Eq4} we get 
\begin{align*}
(p_1-1)\ldots(p_K-1)\le 2^{2^K}-2^{2^{K-1}}.
\end{align*}
Since $p-1>\frac{p}{2}$ for all primes $p\ge3$, we have 
\begin{align*}
n=p_1\cdots p_K<2^{2^K+K}-2^{2^{K-1}+K}
\end{align*}
\end{proof}

\begin{proof}[Proof of Theorem 1.3]

We follow closely an argument in \cite{HL}. Let 
\begin{align*}
P(X)=X^d+a_1X^{d-1}+\ldots+a_d\in\mathbb{Z}[X]
\end{align*}
with $d\ge1$. Suppose $n$ is a Deaconescu number. It is known that there exists a positive constant $c$ such that $S_2(n)\ge\frac{cn}{(\log\log{3n})^2}$ for all odd $n$ (see \cite{Y}). Then 
\begin{align}
\label{Eq5}
M\ll (\log\log{n})^2.
\end{align}
Since $P(S_2(n))\equiv0\ (\text{mod} \ \phi(n))$ we have that $M^dP(S_2(n))\equiv0\ (\text{mod} \ \phi(n))$. Thus by \eqref{Eq1}, we get 
\begin{align*}
(-1)^d+a_1M(-1)^{d-1}+\ldots+a_dM^d\equiv0(\text{mod}\ \phi(n)).
\end{align*}
Let $S$ denote the left hand side of the above congruence. Now we consider two case:\\ 
{\it Case I}. $S\ne0$. Note that $\phi(n)\ge\sqrt{n}$ for all odd $n$. Then from the above congruence and \eqref{Eq5}, we have that 
\begin{align*}
\sqrt{n}\le\phi(n)\le|S|<\bigg(1+\sum_{j=1}^d |a_j|\bigg)M^d\ll(\log\log{n})^{2d}
\end{align*}
which implies $n\ll1$, as we want.\\ 
{\it Case II}. $S=0$. Then $(-1)^d+a_1M(-1)^{d-1}+\ldots+a_dM^d=0$ or 
\begin{align*}
\bigg(-\frac{1}{M}\bigg)^d+a_1\bigg(-\frac{1}{M}\bigg)^{d-1}+\ldots+a_d=0
\end{align*}
or $P\bigg(-\frac{1}{M}\bigg)=0$. Thus we get that $-\frac{1}{M}$ is both an algebraic integer and a rational number which is impossible since $M\ge3$.
\end{proof}


\end{document}